\documentclass[11pt]{amsart}
\usepackage{amsmath, amssymb, amscd, mathrsfs, url, pinlabel,verbatim}
\usepackage[pagebackref]{hyperref}
\usepackage[margin=1in,marginparwidth=0.75in,centering,letterpaper,dvips]{geometry}
\usepackage{color,dcpic,latexsym,graphicx,epstopdf,comment}
\usepackage[all]{xy}
\usepackage[dvipsnames]{xcolor}
\usepackage{tikz,tikz-cd,pgfplots}
\usepackage{upquote,tabularx,textcomp}
\usepackage[shortlabels]{enumitem}
\usepackage{mathtools}
\usepackage[color=blue!20!white,textsize=tiny]{todonotes}

\title{Ribbon concordance and fibered predecessors}

\author[John A. Baldwin]{John A. Baldwin}
\address{Department of Mathematics \\ Boston College}
\email{john.baldwin@bc.edu}

\author[Steven Sivek]{Steven Sivek}
\address{Department of Mathematics\\Imperial College London}
\email{s.sivek@imperial.ac.uk}

\makeatletter
\newtheorem*{rep@theorem}{\rep@title}
\newcommand{\newreptheorem}[2]{%
\newenvironment{rep#1}[1]{%
 \def\rep@title{#2 \ref{##1}}%
 \begin{rep@theorem}}%
 {\end{rep@theorem}}}
\makeatother

\newtheorem {theorem}{Theorem}
\newreptheorem{theorem}{Theorem}
\newtheorem {lemma}[theorem]{Lemma}
\newtheorem {proposition}[theorem]{Proposition}
\newtheorem {corollary}[theorem]{Corollary}

\newtheorem {question}[theorem]{Question}
\newtheorem {problem}[theorem]{Problem}

\numberwithin{equation}{section}
\numberwithin{theorem}{section}

\theoremstyle{definition}

\newtheorem{remark}[theorem]{Remark}
\newtheorem*{remark*}{Remark}

\newlist{pcases}{enumerate}{1}
\setlist[pcases]{
  label=\bf{Case~\arabic*:}\protect\thiscase.~,
  ref=\arabic*,
  align=left,
  labelsep=0pt,
  leftmargin=0pt,
  labelwidth=0pt,
  parsep=0pt
}
\newcommand{\case}[1][]{%
  \if\relax\detokenize{#1}\relax
    \def\thiscase{}%
  \else
    \def\thiscase{~#1}%
  \fi
  \item
}

\newcommand{\Z}{\mathbb{Z}}

\newcommand{\Q}{\mathbb{Q}}

\renewcommand{\phi}{\varphi}

\newcommand\HFKt{\widetilde{\mathit{HFK}}}

\newcommand\HFK{\widehat{\mathit{HFK}}}

\DeclareFontFamily{U}{mathx}{\hyphenchar\font45}
\DeclareFontShape{U}{mathx}{m}{n}{
      <5> <6> <7> <8> <9> <10>
      <10.95> <12> <14.4> <17.28> <20.74> <24.88>
      mathx10
      }{}
\DeclareSymbolFont{mathx}{U}{mathx}{m}{n}
\DeclareFontSubstitution{U}{mathx}{m}{n}
\DeclareMathAccent{\widecheck}{0}{mathx}{"71}

\newcommand{\tr}{\operatorname{tr}}

\makeatletter
\DeclareFontFamily{OMX}{MnSymbolE}{}
\DeclareSymbolFont{MnLargeSymbols}{OMX}{MnSymbolE}{m}{n}
\SetSymbolFont{MnLargeSymbols}{bold}{OMX}{MnSymbolE}{b}{n}
\DeclareFontShape{OMX}{MnSymbolE}{m}{n}{
    <-6>  MnSymbolE5
   <6-7>  MnSymbolE6
   <7-8>  MnSymbolE7
   <8-9>  MnSymbolE8
   <9-10> MnSymbolE9
  <10-12> MnSymbolE10
  <12->   MnSymbolE12
}{}
\DeclareFontShape{OMX}{MnSymbolE}{b}{n}{
    <-6>  MnSymbolE-Bold5
   <6-7>  MnSymbolE-Bold6
   <7-8>  MnSymbolE-Bold7
   <8-9>  MnSymbolE-Bold8
   <9-10> MnSymbolE-Bold9
  <10-12> MnSymbolE-Bold10
  <12->   MnSymbolE-Bold12
}{}

\let\llangle\@undefined
\let\rrangle\@undefined
\DeclareMathDelimiter{\llangle}{\mathopen}%
                     {MnLargeSymbols}{'164}{MnLargeSymbols}{'164}
\DeclareMathDelimiter{\rrangle}{\mathclose}%
                     {MnLargeSymbols}{'171}{MnLargeSymbols}{'171}
\makeatother

\newcounter{desccount}

\newcommand{\descref}[1]{\hyperref[#1]{#1}}

\usetikzlibrary{calc,intersections}
\tikzset{every picture/.style=thick}
\tikzset{link/.style = { white, double = black, line width = 1.75pt, double distance = 1.25pt, looseness=1.75 }}
\tikzset{crossing/.style = {draw, circle, dotted, minimum size=0.5cm, inner sep=0, outer sep=0}}
\pgfplotsset{compat=1.12}

\begin{document}

\begin{abstract}
Given any knot $K\subset S^3$ we prove that there are only finitely many hyperbolic fibered knots which are ribbon concordant to $K$. It follows  that every fibered knot in $S^3$ has only finitely many hyperbolic predecessors under ribbon concordance. Our proof combines results about  maps on Floer homology induced by ribbon  cobordisms with a relationship between the knot Floer homology of a fibered knot and fixed points of its monodromy. We then use the same  techniques in combination with results of Cornish and Kojima--McShane to prove an inequality relating the  volumes of ribbon concordant hyperbolic fibered knots.

\end{abstract}

\maketitle

\section{Introduction}
\label{sec:intro}

Given knots $J$ and $K$ in $S^3$, a \emph{ribbon concordance from $J$ to $K$} is a smooth concordance $C\subset S^3\times I$ from $J$ to $K$  such that the projection \[S^3\times I \to I\] restricts to a Morse function on $C$ with no index 2 critical points. If such a concordance exists, we write $J\leq K$ and say that \emph{$J$ is ribbon concordant to $K$.} This relation was introduced by Gordon in \cite{gordon-ribbon}, where he conjectured that it defines a partial ordering on knots in $S^3$. Gordon's conjecture was only recently proven by Agol in \cite{agol-ribbon}. %In \cite{zemke-ribbon}, Zemke proved that ribbon concordances induce  injections on  knot Floer homology,  providing  strong evidence for Gordon's conjecture, which was  shortly thereafter  settled  by Agol    via different means \cite{agol-ribbon}.

In the same paper, Gordon   asked \cite[Question~6.2]{gordon-ribbon} whether for every sequence of knots \[\dots \leq K_3\leq K_2 \leq K_1\]  there is some $m$ such that $K_n = K_m$ for all  $n\geq m$. This question has an affirmative answer when $K_1$ is fibered, or nearly fibered of genus one (Lemma \ref{lem:explain}). Moreover, Zemke's work \cite{zemke-ribbon} implies that the knot Floer homologies of the  knots in any such sequence must eventually stabilize (their Khovanov homologies too, by \cite{levine-zemke-ribbon}), but Gordon's question is wide open in general.

In this paper, we pose and study the following even more basic question,  an affirmative answer to which would imply an affirmative answer to Gordon's question above:

\begin{question}
\label{ques:finite} For each knot $K\subset S^3$, are there only finitely many $J\leq K$?
\end{question}

Our main theorem makes progress on this question, stating that each knot in $S^3$ has only finitely many \emph{hyperbolic fibered} predecessors under ribbon concordance:

\begin{theorem}
\label{thm:main1}
For each knot $K\subset S^3$, there are only finitely many hyperbolic fibered $J\leq K$.
\end{theorem}

Theorem \ref{thm:main1} implies, via work of Silver \cite{silver} and Kochloukova \cite{kochloukova}, that each fibered knot has only finitely many \emph{hyperbolic} predecessors,   nearly answering Question \ref{ques:finite} when $K$ is fibered:

\begin{corollary}
\label{cor:main1}
For each fibered knot $K\subset S^3$, there are only finitely many hyperbolic  $J\leq K$. 
\end{corollary}

Our proof of Theorem \ref{thm:main1} in the next section uses results about the injectivity of  Heegaard Floer invariants under ribbon concordance and  ribbon $\Z/2$-homology cobordism, together with the fact that the knot Floer homology of a fibered knot bounds the number of fixed points of its monodromy, to prove that the dilatation of a pseudo-Anosov representative of the monodromy of $J$ is bounded above by the factorial of the arc index of $K$ (Proposition \ref{prop:dilatation}). We then use the fact that there are only finitely many conjugacy classes of pseudo-Anosov homeomorphisms of a fixed surface with dilatation less than a given constant to show that there are only finitely many possible $J \leq K$.

The bound in Proposition \ref{prop:dilatation}, combined with work of Kojima--McShane \cite[Theorem 1]{kojima-mcshane} on  the relationship between the entropy of a pseudo-Anosov homeomorphism (the log of its dilatation) and the volume of its mapping torus, also shows   that the genus and arc index of a knot give rise to a bound on the volume of any  hyperbolic fibered predecessor:

\begin{theorem}
\label{thm:vol-arc} Suppose that $K\subset S^3$ is a knot of genus $g$ and arc index $\delta$. Then \[\mathrm{vol}(S^3\setminus J) \leq 3\pi(2g-1)\log(\delta!)\] for every hyperbolic fibered  $J\leq K$.
\end{theorem}

When $K$ is also hyperbolic and fibered, we use the same approach, together with work of Cornish \cite{cornish}, to bound the dilatation of the monodromy of $J$ in terms of that of  $K$ (Lemma \ref{lem:entropy}). We then apply the results of  Kojima \cite{kojima} and Kojima--McShane \cite{kojima-mcshane}   to relate the volumes of two ribbon concordant hyperbolic fibered knots:

\begin{theorem}\label{thm:volume} For each natural number $g$ and real number $\epsilon>0$, there exists a constant $c_{g,\epsilon}$ such that if $J$ and $K$ are hyperbolic fibered knots in $S^3$ with $J\leq K$, then \[\mathrm{vol}(S^3\setminus J) \leq c_{g,\epsilon}\cdot \mathrm{vol}(S^3\setminus K), \] whenever $K$ has genus $g$ and the systole of $S^3\setminus K$ is at least $\epsilon$.
\end{theorem}

These results  bear on another question  posed  by Gordon in \cite[Question 6.4]{gordon-ribbon}, which asks whether simplicial volume is monotonic under ribbon concordance.

\begin{remark}%Finally, we note that all of these results hold in an \emph{a priori}  more general setting.
A \emph{handle-ribbon concordance}, also known as a \emph{strongly homotopy-ribbon concordance}, is a smooth concordance in $S^3\times I$ from $J$ to $K$ whose complement can be obtained from $S^3\setminus J$ by attaching $1$-handles and $2$-handles. Ribbon concordances are handle-ribbon but it is open whether the converse holds. The results above   remain true, via the same  proofs, if ``$J\leq K$" is replaced  in their statements by ``$J$ is handle-ribbon concordant to $K$." 
\end{remark}

\begin{remark}
Ian Agol tells  us  that he has since found a different approach to some of these results. In particular, he can prove that simplicial volume is monotonic under ribbon concordances between fibered knots, settling Gordon's question in that case.
\end{remark}

\subsection{Organization} 
We prove these results in \S\ref{sec:proofs}. In \S\ref{sec:comment}, we discuss   questions related to this work.  %For instance, can we remove \emph{hyperbolic} from Theorem \ref{thm:main1}? And can we hope to say anything about  the finiteness of \emph{nonfibered} predecessors $J$ via similar methods? 

\subsection{Acknowledgements} We thank Ian Agol, Jen Hom,  Tye Lidman, Robert Lipshitz, Yi Liu, and Chi Cheuk Tsang for helpful and interesting conversations. JAB was supported by NSF Grant DMS-2506250.  SS was supported by the Engineering and Physical Sciences Research Council [grant number UKRI1016].  No data were created or analyzed in this work.

\section{Proofs}\label{sec:proofs}
Let us set  the following notation and conventions:
\begin{itemize}
\item all Floer homology groups below are with coefficients in $\Z/2$;
\item given a knot $K\subset S^3$, let $K_n$ denote the lift of $K$ to its $n$-fold cyclic branched cover $ \Sigma_n(K)$;
\item given a hyperbolic fibered knot $K\subset S^3$, let $\lambda(K)$ denote the dilatation of a pseudo-Anosov representative of its monodromy.
\end{itemize}
We begin by bounding the knot Floer homology of $K_n$ in terms of the arc index of $K$. 

\begin{lemma}
\label{lem:inequality}
Suppose that $K\subset S^3$ is a knot with arc index $\delta$. Then \[\dim \HFK(\Sigma_n(K),K_n) \leq \frac{(\delta!)^n}{2^{\delta -1}} \] for all natural numbers $n$.
\end{lemma}

\begin{proof}
Levine shows in \cite{levine-cyclic} that $K_n\subset \Sigma_n(K)$ can be represented by a multipointed Heegaard diagram $\mathcal{H} = (S, \alpha,\beta,\mathbf{z},\mathbf{w})$ with \[|
\alpha| = |\beta| = \delta n \quad \textrm{ and } \quad |\mathbf{z}| = |\mathbf{w}| = \delta,\] obtained as an $n$-fold cyclic branched cover of a grid diagram for $K$. Moreover, each curve in $\alpha$ has $\delta$ total intersections with the curves in $\beta$, and vice versa. It follows that the associated Heegaard Floer complex, and therefore its homology $\HFKt(\mathcal{H}),$ has dimension at most $(\delta!)^n$. This is obvious from a closer inspection of Levine's diagrams, but also follows from a  general fact: generators are perfect matchings of a bipartite graph with $\alpha$ circles on one side and $\beta$ curves on the other, with an edge for each point of $\alpha \cap \beta$; there are $2\delta n$ vertices, each with degree $\delta$, so there are at most $(\delta!)^n$ perfect matchings by  the main result of \cite{alon-friedland}. Finally, \[\dim \HFK(\Sigma_n(K),K_n) = \frac{1}{2^{\delta-1}}\cdot \dim \HFKt(\mathcal{H})\] \cite[Theorem~1.1]{levine-cyclic}, from which the result follows.
\end{proof}

\begin{lemma} \label{lem:branched-cover}
Suppose  $C \subset S^3\times I$ is  a ribbon concordance from $J$ to $K$.  Let
\[ W_n : \Sigma_n(J) \to \Sigma_n(K) \]
be the $n$-fold branched cyclic cover of $S^3\times I$ branched along $C$, and let $C_n$ be the lift of $C$ to $W_n$.  Then $W_n\setminus \nu(C_n)$ is a ribbon $\Z/p$-homology cobordism whenever $n$ is a power of a prime $p$.
\end{lemma}

\begin{proof}
Since $C$ is a ribbon concordance,  its complement is a ribbon cobordism; that is, $S^3\times I \setminus \nu(C)$ is obtained from $S^3\setminus \nu(J)$ by attaching 1-handles and 2-handles  \cite{gordon-ribbon}. The $n$-fold cyclic cover $W_n\setminus \nu(C_n)$ is therefore also a ribbon cobordism, as we can just lift the handle decomposition.

It remains to show that if $n=p^e$ is a prime power, then $W_n\setminus \nu(C_n)$ is a $\Z/p$-homology cobordism \[\Sigma_n(J)\setminus \nu(J_n) \to \Sigma_n(K)\setminus \nu(K_n);\] in other words, that the inclusions of these knot complements into the complement of $C_n$ induce isomorphisms on homology with $\Z/p$-coefficients.  These manifolds are the $n$-fold cyclic covers of the complements of $C$ (in $S^3 \times I$) and of $J$ and $K$ (in $S^3$), each of which has the homology of $S^1\times D^2$.  Now if $X$ is a homology solid torus, and $X_k$ and $X_\infty$ are its $k$-fold and infinite cyclic covers, we follow \cite[\S5]{gordon-aspects} to get for any power $n=p^e$ an exact sequence
\[ H_1(X_\infty; \Z/p) \xrightarrow{t^{p^e}-1} H_1(X_\infty; \Z/p) \to H_1(X_{p^e}; \Z/p) \to \Z/p \to 0 \]
of $(\Z/p)[t,t^{-1}]$-modules, where $t$ acts by a generator of the deck transformation group on each $H_1$ and trivially on the last $\Z/p$.  Since \[(t^{p^e}-1) \equiv (t-1)^{p^e} \pmod{p},\] and $t-1$ acts surjectively on $H_1(X_\infty;\Z/p)$ (see \cite[\S4]{gordon-aspects}), it follows that $H_1(X_{p^e}; \Z/p) \cong \Z/p$ for all $e \geq 1$.  In particular, each of the inclusion-induced maps
\begin{align*}
H_1(\Sigma_n(J) \setminus \nu(J_n); \Z/p) &\to H_1(W_n \setminus \nu(C_n); \Z/p), \\
H_1(\Sigma_n(K) \setminus \nu(K_n); \Z/p) &\to H_1(W_n \setminus \nu(C_n); \Z/p)
\end{align*}
is a map from $\Z/p$ to $\Z/p$, which is then an isomorphism because each $H_1$ is generated by a meridian of $J_n$ or  $K_n$. (To see this for $W_n \setminus \nu(C_n)$, note that the meridian of $J$ or $K$ is also a meridian of $C$, hence is dual to some relative cycle in $H_3((S^3\times I) \setminus \nu(C), \partial)$, and that the corresponding meridian of $J_n$ or $K_n$ in $W_n$ is nonzero in homology because it is dual to a lift of that cycle.)

Finally, each of these spaces has $H_k = 0$ over $\Z/p$ for  $k \geq 2$: we've seen that $H_0 \cong H_1 \cong \Z/p$; they all have $H_k = 0$ for $k \geq 3$, since these homology groups vanish for $\Sigma_n(J) \setminus \nu(J_n)$ and for $\Sigma_n(K) \setminus \nu(K_n)$, and since $W_n \setminus \nu(C_n)$ is obtained from the complement of $J_n$ by attaching only $1$-handles and $2$-handles; and then since their Euler characteristics are zero (as covers of spaces with trivial Euler characteristic) they have $H_2 = 0$ as well.  We conclude that $W_n \setminus \nu(C_n)$ is a $\Z/p$-homology cobordism as claimed.
\end{proof}

The importance of Lemma \ref{lem:branched-cover} is that it enables us to prove the following, which is a consequence of the  behavior of Heegaard Floer homology under ribbon $\Z/2$-homology cobordisms:

\begin{corollary}
\label{cor:rank} Suppose that $J$ and $K$ are knots in $S^3$ with $J\leq K$. Then \[\dim \HFK(\Sigma_n(J),J_n) \leq  \dim \HFK(\Sigma_n(K),K_n)\] whenever $n$ is a power of $2$.
\end{corollary}

\begin{proof}
Suppose  $n$ is a power of 2, and let $C$ be a ribbon concordance from $J$ to $K$. Lemma \ref{lem:branched-cover} says that the exterior of the lifted concordance $C_n$  from $J_n$ to $K_n$ is a ribbon $\Z/2$-homology cobordism. Then \cite[Corollary 4.13]{dlvvw} asserts that the knot Floer homology of $J_n$ is a direct summand of the knot Floer homology of $K_n$.
\end{proof}

The following lemma is well-known and often credited to Thurston \cite{thurston-surfaces}, but we've included a proof  since we could not find a reference quite to our liking. In particular, several references (like \cite{flp,ivanov-entropy}) express the dilatation in terms of a $\limsup$, which is not sufficient for our needs, since in the proof of Proposition \ref{prop:dilatation} we have to take a limit as $n$ goes to infinity along powers of 2.

\begin{lemma}
\label{lem:fix} Suppose  $K\subset S^3$ is a hyperbolic fibered knot, and let $\varphi$ be a pseudo-Anosov representative of its monodromy. Then the dilatation $\lambda(K)$ is given by \[\lambda(K) = \lim_{n\to \infty} \big(\#\mathrm{Fix}(\varphi^n)\big)^{1/n},\] where $\#\mathrm{Fix}(\varphi^n)$ denotes the number of fixed points of $\varphi^n$.
\end{lemma}

\begin{proof}
Fix an invariant train track for $\varphi$, and let $M$ be the incidence matrix for the train track map induced by $\varphi$. Then the difference between $\#\mathrm{Fix}(\varphi^n)$ and $\tr(M^n)$ is bounded by a constant multiple of the genus of $K$, as explained for example in \cite[Section 5]{cotton-clay}. Now, $M$ is Perron--Frobenius, and the dilatation $\lambda(K)$ is the unique eigenvalue of $M$ of maximal modulus. It follows that \[\lambda(K) = \lim_{n\to \infty} \big(\tr(M^n)\big)^{1/n}= \lim_{n\to \infty} \big(\#\mathrm{Fix}(\varphi^n)\big)^{1/n}\] as desired.
\end{proof}

\begin{proposition}
\label{prop:dilatation}
Suppose that $K\subset S^3$ is a knot with arc index $\delta$. If $J\subset S^3$ is a fibered hyperbolic knot with $J\leq K$, then $\lambda(J) \leq \delta!$.
\end{proposition}

\begin{proof}
Suppose that $J\leq K$, and let $n= 2^e$ be a power of 2. Then  \[\dim \HFK(\Sigma_n(J),J_n) \leq  \dim \HFK(\Sigma_n(K),K_n) \leq \frac{(\delta!)^n}{2^{\delta -1}},\]
where the first inequality is Corollary \ref{cor:rank} and the second is Lemma \ref{lem:inequality}. Letting \[g = g(J) = g(J_n),\] it follows that \begin{equation}\label{eqn:inequality} \big(\dim \HFK(\Sigma_n(J),J_n,g-1) \big)^{1/n}\leq \frac{\delta!}{2^{(\delta-1)/n}}.\end{equation} Let $\varphi$ be the pseudo-Anosov representative of the monodromy of $K$. Then $\varphi^n$ is the pseudo-Anosov representative of the monodromy of $J_n$, and the number of fixed points of $\varphi^n$ satisfies \[\#\textrm{Fix}(\varphi^n)\leq\dim \HFK(\Sigma_n(J),J_n,g-1)-1,\] by the  work of Ni \cite{ni-fixed} and Ghiggini--Spano \cite{ghiggini-spano}. The inequality \eqref{eqn:inequality}  therefore implies that \[\big (\# \textrm{Fix}(\varphi^n)\big)^{1/n} < \frac{\delta!}{2^{(\delta-1)/n}}.\]  On the other hand, the dilatation $\lambda(J)$  is given by \[\lambda(J) = \lim_{n\to \infty} \big (\# \textrm{Fix}(\varphi^n)\big)^{1/n},\] by Lemma \ref{lem:fix}.
So, letting $e$ and thus $n=2^e$ approach $\infty$, we find that $\lambda(J) \leq \delta!$.
\end{proof}

The same argument combined with a result of Cornish \cite{cornish} proves the following:

\begin{lemma}
\label{lem:entropy}
If  $J\leq K$ are hyperbolic fibered knots in $S^3$ then $\lambda(J) \leq \lambda(K)^{g(K)}.$
\end{lemma}

\begin{proof}
In his thesis \cite[Theorem 3.1.2]{cornish}, Cornish showed   that if $K\subset S^3$ is a  fibered hyperbolic knot, then there is some constant $c$ depending on $K$ such that \[\dim \HFK(\Sigma_n(K),K_n)\leq c(\lambda(K))^{g(K)n}\] for all  natural numbers $n$. Replacing $\frac{(\delta!)^n}{2^{\delta -1}}$ with $c(\lambda(K))^{g(K)n}$ everywhere in the proof of Proposition \ref{prop:dilatation} immediately yields the result.
\end{proof}

We are now ready to prove the  results stated in the introduction.

\begin{proof}[Proof of Theorem \ref{thm:main1}]
Let $K\subset S^3$ be any knot, and let $\delta$ be its arc index. Suppose that $J\subset S^3$ is a fibered hyperbolic knot with $J\leq K$, and let $\varphi$ be a pseudo-Anosov representative of its monodromy. The complement $S^3\setminus \nu(J)$ is determined by the conjugacy class of $\varphi$, and therefore so is $J$ by  \cite{gordon-luecke-complement}. So it suffices to prove that there are only finitely many possibilities for the conjugacy class of $\varphi.$  The dilatation of $\varphi$ satisfies $\lambda(J) \leq \delta!$ by Proposition \ref{prop:dilatation}, and $g(J) \leq g(K)$ by  \cite[Theorem 1.3]{zemke-ribbon}. Theorem \ref{thm:main1} then follows from the fact that for a fixed compact, oriented, connected surface $S$ and a fixed constant $M$, there are only finitely many conjugacy classes of pseudo-Anosov homeomorphisms of $S$ with dilatation at most $M$ \cite{ivanov-bounded-dilatation}. \end{proof}

\begin{proof}[Proof of Corollary \ref{cor:main1}] Suppose  $K\subset S^3$ is fibered, and let $J\leq K$. An observation of Silver  \cite{silver} together with Kochloukova's resolution of Rapaport's conjecture on knot-like groups \cite{kochloukova}  shows that $J$ is fibered as well, as explained in \cite{miyazaki}. The corollary then follows from Theorem \ref{thm:main1}.
\end{proof}

\begin{proof}[Proof of Theorem \ref{thm:vol-arc}] Suppose $K\subset S^3$ has genus $g$ and arc index $\delta$, and let $J\leq K$ be a hyperbolic fibered knot.
Kojima--McShane's result in \cite[Theorem 1]{kojima-mcshane} implies  that 
\begin{equation}\label{eqn:J}\mathrm{vol}(S^3\setminus J) \leq 3\pi(2g(J)-1)\log(\lambda(J)).\end{equation} Then the theorem follows from the inequalities
\[3\pi(2g(J)-1)\log(\lambda(J))\,\leq \,3\pi(2g-1)\log(\lambda(J))\,\leq\, 3\pi (2g-1)\log(\delta!),\]
 coming from Zemke's result \cite{zemke-ribbon} that $g(J)\leq g(K)=g$, and Proposition \ref{prop:dilatation}, respectively.
\end{proof}

\begin{proof}[Proof of Theorem \ref{thm:volume}]
Kojima proves in \cite[Theorem 1]{kojima} that for each natural number $g$ and real number $\epsilon>0$, there is a constant $b_{g,\epsilon}$ such that \begin{equation}\label{eqn:K}\log(\lambda(K)) \leq b_{g,\epsilon}\cdot\mathrm{vol}(S^3\setminus K)\end{equation} for any hyperbolic fibered knot  $K$ of genus $g$  where the  systole of $S^3\setminus K$  is at least $\epsilon$.
On the other hand, the Kojima--McShane result \cite[Theorem 1]{kojima-mcshane} says that 
\begin{equation}\label{eqn:J}\mathrm{vol}(S^3\setminus J) \leq 3\pi(2g(J)-1)\log(\lambda(J))\end{equation} for any hyperbolic fibered knot $J$, as mentioned above. Let \[c_{g,\epsilon} := 3\pi g(2g-1)b_{g,\epsilon}.\] Now, suppose that $J$ and $K$ are hyperbolic fibered knots in $S^3$ with $J\leq K$, where $g(K)=g$ and the systole of $S^3\setminus K$ is at least $\epsilon$. Then
\begin{align*}
\mathrm{vol}(S^3\setminus J) &\leq 3\pi(2g(J)-1)\log(\lambda(J))\\
&\leq 3\pi(2g(K)-1)\log(\lambda(J))\\
&\leq 3\pi g(K)(2g(K)-1)\log(\lambda(K))\\
&= 3\pi g(2g-1)\log(\lambda(K))\\
&\leq 3\pi g(2g-1)b_{g,\epsilon}\cdot\mathrm{vol}(S^3\setminus K)\\
&= c_{g,\epsilon}\cdot\mathrm{vol}(S^3\setminus K),
\end{align*}
where the first line is \eqref{eqn:J}, the second is the fact that $g(J)\leq g(K)$ by Zemke \cite{zemke-ribbon}, the third is Lemma \ref{lem:entropy}, and the fifth is \eqref{eqn:K}.
\end{proof}

\section{Discussion}\label{sec:comment}
In this section, we collect   some comments and questions related to this work. 
%Second, we'd like to remove the \emph{hyperbolic} hypothesis in Theorem \ref{thm:main1}. Fix $K\subset S^3$ and suppose  $J\leq K$ is fibered. If $J$ is not hyperbolic then it's either a torus knot or a satellite knot. Finiteness holds in the first case, since $g(J)\leq g(K)$ by Zemke \cite{zemke-ribbon}, and there are finitely many torus knots of a given genus. In the second case, $J$ has reducible monodromy $\varphi$. If $\varphi$ has  no pseudo-Anosov components, then $J$ is obtained via taking iterated cables and connected sums of torus knots, and there are  finitely many such knots of a given genus. If $\varphi$ has pseudo-Anosov components, then one can hope to bound the dilatations of these components and use that to prove finiteness, as in the proofs of Proposition \ref{prop:dilatation} and Theorem \ref{thm:main1}. To bound dilatations, the naive strategy following Proposition \ref{prop:dilatation} would be to first take an $m$-fold  cyclic cover of the concordance complement such that $\varphi^m$ fixes each pseudo-Anosov component setwise, and then  consider further $2^e$-fold cyclic covers from there. The issue is that the resulting covers, while still ribbon cobordisms, will not necessarily be  $\Z/2$- or even $\Q$-homology cobordisms, which means that one cannot apply the injectivity results of \cite{dlvvw} that are used for Corollary \ref{cor:rank} and subsequently Proposition \ref{prop:dilatation}.

First, we'd like to remove the \emph{hyperbolic} hypothesis in Theorem \ref{thm:main1}. Fix $K\subset S^3$ and suppose that  $J\leq K$ is fibered. If $J$ is not hyperbolic then either:
\begin{itemize}
\item $J$ is a torus knot, and there are only finitely many possible such $J$, since $g(J)\leq g(K)$ by Zemke \cite{zemke-ribbon}, and there are finitely many torus knots of a given genus; or
\item $J$ is a satellite knot, and  has reducible monodromy $\varphi$. If moreover $\varphi$ has no pseudo-Anosov components, then $J$ is obtained by taking iterated cables and connected sums of torus knots, and again there are finitely many such knots of a given genus.
\end{itemize}
If $J$ has reducible monodromy $\varphi$ with pseudo-Anosov components, then one can hope to bound the dilatations of these components and use that to prove finiteness, as in the proofs of Proposition \ref{prop:dilatation} and Theorem \ref{thm:main1}. To bound these dilatations, the naive strategy following Proposition \ref{prop:dilatation} would be to first take an $m$-fold  cyclic cover of the concordance complement such that $\varphi^m$ fixes each pseudo-Anosov component setwise, and then  consider further $2^e$-fold cyclic covers from there. The issue is that the resulting covers, while still ribbon cobordisms, will not necessarily be  $\Z/2$- or even $\Q$-homology cobordisms, which means that one cannot apply the injectivity results of \cite{dlvvw} that are used for Corollary \ref{cor:rank} and subsequently Proposition \ref{prop:dilatation}.

Second, and  more importantly, we'd like to remove the \emph{fibered} hypothesis in Theorem \ref{thm:main1}, so as to answer Question \ref{ques:finite}, and therefore Gordon's question about descending sequences of ribbon concordances, in full. To prove the finiteness of hyperbolic \emph{nonfibered} $J\leq K$, one might first hope to  bound, in terms of a quantity depending only on  $K$,  the entropy of Gabai--Mosher pseudo-Anosov flows  transverse to a minimal genus Seifert surface for $J$. One issue is that we do not  yet understand the relationship between knot Floer homology and the dynamics of such flows in the nonfibered case. Another is that, unlike in the fibered case, it's not true that for a fixed genus there are only finitely many such flows, up to orbit equivalence, with entropy less than a given constant. However, such a finiteness result may be possible if we also have a bound on the topological complexity of the sutured Seifert surface complement. This raises the    following vague question about the behavior of \emph{guts}  (see \cite{agol-zhang})  under ribbon concordance:

\begin{question}
Does $J\leq K$ imply that the guts of $J$  are no more complicated than those of $K$?
\end{question}

%When $J\leq K$ and $g(J) = g(K)$,  Zemke's result \cite{zemke-ribbon}   that  \[\dim\HFK(S^3,J,g(J)) \leq \dim\HFK(S^3,K,g(K))\]  provides some hope of an  answer to this  vague question, as these knot Floer groups contain some information about guts (such as whether they are empty).
Here's a more precise version,  relevant to Gordon's question:

\begin{question}
Is it true that for any sequence of knots \[\dots \leq K_3\leq K_2 \leq K_1\] there is some $m$ such that the guts of $K_n$ are the same as those of $K_m$ for all $n\geq m$?  
\end{question}

Recall that a knot $K\subset S^3$ is said to be \emph{nearly fibered} if \[\dim \HFK(S^3,K,g(K)) = 2.\] We classified the nearly fibered knots of genus one  in \cite{bs-nearly}; they include the knot $5_2$ and the pretzel knots $P(-3,3,2n+1)$. Moreover,  Li--Ye proved  in \cite{li-ye-nearly} that there are only three possibilities for the guts of any nearly fibered knot, so perhaps the following is   tractable:

\begin{problem}
Prove for any knot $K\subset S^3$ that there are only finitely many  nearly fibered  $J\leq K$.
\end{problem}

On the subject of Gordon's question  and nearly fibered knots, here is the lemma we promised in \S\ref{sec:intro}. We learned  it from Jen Hom and don't claim any  originality:

\begin{lemma}\label{lem:explain} Suppose that either $K_1$ is  fibered, or  $K_1$ is nearly fibered of  genus one. Then for any sequence of knots satisfying \[\dots\leq K_3\leq K_2\leq K_1\] there is some $m$ such that $K_n = K_m$ for all $n\geq m$.
\end{lemma}

\begin{proof}
Suppose first that $K_1$ is fibered. Then the work of Silver  \cite{silver} and Kochloukova \cite{kochloukova} implies that each $K_i$ is fibered, as noted in the proof of Corollary \ref{cor:main1}.  Gordon's work \cite[Lemma 3.4(i)]{gordon-ribbon}  implies that the  degrees of the Alexander polynomials of these knots    stabilize. Finally,  since ribbon concordant fibered knots are equal  whenever their Alexander polynomials have the same degree \cite[Lemma 3.4(ii)]{gordon-ribbon}, we conclude that  there is some $m$ such that $K_n = K_m$ for $n\geq m$. 

Suppose next that $K_1$ is nearly fibered of genus one, so that \[\dim\HFK(S^3,K_1,1) =2.\] The work of Zemke \cite{zemke-ribbon} and Levine--Zemke \cite{levine-zemke-ribbon} implies that there is some $m$ such that the knot Floer  and Khovanov homologies of $K_n$ agree with those of $K_m$ for all $n\geq m$. If the knot $K_m$ is fibered, then we are done as above. If not, then Zemke's work  \cite{zemke-ribbon}, together with Ghiggini's fiberedness detection result for genus one knots \cite{ghiggini-fibered}, implies that \[2\leq \dim\HFK(S^3,K_m,1) \leq \dim\HFK(S^3,K_1,1) = 2,\] from which it follows that  all $K_n$ with $n\geq m$ are   nearly fibered of genus one. But our classification of nearly fibered genus one knots in \cite{bs-nearly} shows that there are not infinitely many distinct such knots with the same Khovanov homologies.
\end{proof}

%Finally, Lemma \ref{lem:entropy} suggests the question of whether dilatation, like so many other invariants, is  monotonic under ribbon concordance:

%\begin{question} 
%Is it  the case that $\lambda(J)\leq \lambda(K)$ for any  hyperbolic fibered knots with $J\leq K$?
%\end{question}

\bibliographystyle{myalpha}
\bibliography{References}

\end{document}